\documentclass[a4paper,10pt]{article}

\usepackage{geometry}

\usepackage{nomencl}
\usepackage{amsmath,amssymb,amstext} 
\usepackage{amsthm, mathrsfs}
\usepackage{authblk}
\usepackage{natbib}
\usepackage{dsfont} 

\newcommand{\ud}{\, \mathrm{d}}

\newcommand{\indicatorfn}{\mathds{1}}

\newcommand{\expectation}{\textrm{E}}
\newcommand{\filtration}{\mathcal{F}}

\newcommand{\Omegaset}{\Omega}

\newcommand{\timehorizon}{T}
\newcommand{\timeset}{[0,\timehorizon]}

\newcommand{\Pmeasure}{\mathbb{P}}
\newcommand{\probspace}{(\Omegaset, \filtration, \Pmeasure)}

\newcommand{\measurablespace}{\Omegaset,\filtration}

\newcommand{\BMdim}{N}
\newcommand{\BM}{W}
\newcommand{\vBM}{\BM}

\newcommand{\markovspace}{I}
\newcommand{\Markovgenerator}{G}
\newcommand{\markovgenerator}{g}
\newcommand{\markovchain}{\alpha}
\newcommand{\markovdim}{D}
\newcommand{\markovinitialstate}{i_{0}}

\newcommand{\markovmartingale}{M}

\newcommand{\markovsquareqvprocess}{N}
\newcommand{\markovcompensator}{\lambda}

\newcommand{\markovangleqvprocessij}{\langle \markovmartingale_{ij} \rangle}

\newcommand{\naturalnumbers}{\mathbb{N}}

\newcommand{\realnumbers}{\mathbb{R}}
\newcommand{\realnumbersN}{\mathbb{R}^{N}}
\newcommand{\realnumbersP}{\mathbb{R}^{P}}
\newcommand{\realnumbersNN}{\mathbb{R}^{N \times N}}
\newcommand{\realnumbersDD}{\mathbb{R}^{D \times D}}

\newcommand{\admissibleportfolios}{\mathcal{A}}

\newcommand{\meanret}{b}

\newcommand{\varianceret}{\sigma}
\newcommand{\varianceretT}{\sigma^{\top}}

\newcommand{\ITOX}{X}
\newcommand{\vITOX}{\ITOX}

\newcommand{\controlset}{\mathcal{U}}
\newcommand{\control}{u}
\newcommand{\vcontrol}{\control}

\newcommand{\vcontrolhat}{\hat{\vcontrol}}
\newcommand{\controlopt}{u^{\star}}
\newcommand{\vcontrolopt}{\controlopt}

\newcommand{\controladmiss}{\mathcal{A}}

\newcommand{\state}{X}
\newcommand{\statehat}{\hat{\state}}
\newcommand{\statezero}{x}
\newcommand{\stateopt}{\state^{\star}}

\newcommand{\hamiltonian}{\mathcal{H}}
\newcommand{\hamiltonianhat}{\hat{\hamiltonian}}
\newcommand{\hamp}{p}
\newcommand{\hamq}{q}
\newcommand{\hamr}{\eta}
\newcommand{\vhamp}{\hamp}
\newcommand{\vhamq}{\hamq}
\newcommand{\vhamr}{\hamr}
\newcommand{\vhamqT}{\vhamq^{\top}}
\newcommand{\hamhatp}{\hat{\hamp}}
\newcommand{\hamhatq}{\hat{\hamq}}
\newcommand{\hamhatr}{\hat{\hamr}}
\newcommand{\vhamhatp}{\hat{\hamp}}
\newcommand{\vhamhatq}{\hat{\hamq}}
\newcommand{\vhamhatr}{\hat{\hamr}}
\newcommand{\vhamhatqT}{\hat{\hamq}^{\top}}

\newcommand{\valuefn}{V}
\newcommand{\constc}{d}
\newcommand{\initialstate}{x_{0}}

\newcommand{\Jfn}{J}
\newcommand{\Jfncontrol}{J_{\vcontrol}}
\newcommand{\ffn}{f}
\newcommand{\gfn}{h}
\newcommand{\HJBF}{F}
\newcommand{\kwrtx}{k}

\newcommand{\fnYk}{Y_{\kwrtx}}

\newcommand{\fnpsi}{R}
\newcommand{\fnP}{S}
\newcommand{\integrandpsi}{\nu^{\fnpsi}}
\newcommand{\integrandP}{\nu^{\fnP}}
\newcommand{\integrandpsihat}{\hat{\nu}^{\fnpsi}}
\newcommand{\integrandPhat}{\hat{\nu}^{\fnP}}

\newcommand{\adjppsi}{\psi}
\newcommand{\adjpphi}{\phi}

\newcommand{\QLM}{quadratic loss minimization }

\newcommand{\bankaccount}{S_{0}}
\newcommand{\stockprice}{S_{1}}

\theoremstyle{plain}
\newtheorem{thm}{Theorem}[section]

\theoremstyle{definition}

\theoremstyle{remark}

\newtheorem{rmk}[thm]{Remark}

\numberwithin{equation}{section}

\title{Sufficient stochastic maximum principle in a regime-switching diffusion model}

\author{Catherine Donnelly\thanks{({\tt C.A.Donnelly@hw.ac.uk}).  Mailing address: Department of Actuarial Mathematics and Statistics, Heriot-Watt University, Edinburgh, U.K.  Phone: +44 131 451 3251.  Fax: +44 131 451 3249.} \, \emph{Heriot-Watt University, Edinburgh, U.K.}}

\begin{document}

\maketitle

\begin{abstract}
We prove a sufficient stochastic maximum principle for the optimal control of a regime-switching diffusion model.  We show the connection to dynamic programming and we apply the result to a quadratic loss minimization problem, which can be used to solve a mean-variance portfolio selection problem.
\end{abstract}



\pagestyle{myheadings}
\thispagestyle{plain}
\markboth{C. DONNELLY}{STOCHASTIC MAXIMUM PRINCIPLE}

\section{Introduction}

The aim of this paper is to prove a sufficient stochastic maximum principle for optimal control within a regime-switching diffusion model.  This extends the result of \citet{framstadetal04.article}, which is in a jump-diffusion setting.  To prove this, we follow the method in \citet{framstadetal04.article}.  As in their paper, we show the connection to dynamic programming and show how to apply the result to a \QLM problem.

An early maximum principle for a diffusion model is in \citet{bismut73.article}, where a necessary maximum principle is derived in a model which is somewhat structurally similar to our own and, as we also find in our set-up, this results in jumps in the adjoint variables of the Hamiltonian.

For a hidden Markovian regime-switching diffusion model, \citet{elliottsiubadescu10.article} apply, though do not state explicitly, a sufficient maximum principle to a mean-variance portfolio selection problem.  However, their model is not the same as the one we consider and hence they do not obtain jumps in the adjoint variables.

In Section \ref{SECmodel} we detail the regime-switching diffusion model and in Section \ref{SECproblem} we set out the control problem.  The sufficient stochastic maximum principle is given in Section \ref{SECsmp}.  This is followed by demonstrating in Section \ref{SECdynamic} the connection with dynamic programming.  Finally, in Section \ref{SECexample} we illustrate the use of the sufficient stochastic maximum principle by solving a \QLM problem.

\section{The regime-switching diffusion model} \label{SECmodel}
Let $\timehorizon \in (0, \infty)$ be a fixed, deterministic time.  We assume that we are given an $\BMdim$-dimensional Brownian motion $\BM = (\BM_{1}, \ldots, \BM_{\BMdim})$ and a continuous-time, finite state space Markov chain $\markovchain$ defined on the same probability space $(\measurablespace, \Pmeasure)$.

The filtration is generated jointly by the Brownian motion $\vBM$ and the Markov chain $\markovchain$,
\begin{equation} \label{MKTfiltration}
 \filtration_{t} := \sigma \{ (\markovchain (s), \vBM (s)), s \in [0,t] \} \vee \mathcal{N} ( \Pmeasure ), \quad \forall t \in \timeset, 
\end{equation}
where $\mathcal{N} ( \Pmeasure )$ denotes the collection of all $\Pmeasure$-null events in the probability space $\probspace$.

We assume that the Markov chain takes values in a finite state space $\markovspace = \{ 1, \ldots, \markovdim \}$ and it starts in a fixed state $\markovinitialstate \in \markovspace$, so that $\markovchain(0) = \markovinitialstate$, a.s.   The Markov chain $\markovchain$ has a generator $\Markovgenerator$ which is a $\markovdim \times \markovdim$ matrix $\Markovgenerator = (\markovgenerator_{ij})_{i,j=1}^{\markovdim}$.  Denote by $\indicatorfn$ the zero-one indicator function.  Associated with each pair of distinct states $(i,j)$ in the state space of the Markov chain is a point process, or counting process,
\begin{equation} \label{EQNmarkovpointprocess}
\markovsquareqvprocess_{ij} (t) :=
\sum_{0 < s \leq t} \indicatorfn_{ \{ \markovchain (s_{-}) = i \} } \, \indicatorfn_{ \{ \markovchain (s) = j \} }, \quad \forall t \in \timeset.
\end{equation}
The process $\markovsquareqvprocess_{ij} (t)$ counts the number of jumps that the Markov chain $\markovchain$ has made from state $i$ to state $j$ up to time $t$.  Define the intensity process
\begin{equation} \label{EQNintensitymarkovmartingale}
 \markovcompensator_{ij} (t) := \markovgenerator_{ij} \, \indicatorfn_{ \{ \markovchain (t_{-}) = i \} }.
\end{equation}
If we compensate $\markovsquareqvprocess_{ij} (t)$ by $\int_{0}^{t} \markovcompensator_{ij} (s) \ud s$, then the resulting process
\begin{equation} \label{EQNcanonicalmarkovmartingale}
 \markovmartingale_{ij} (t) := \markovsquareqvprocess_{ij} (t) - \int_{0}^{t} \markovcompensator_{ij} (s) \ud s
\end{equation}
is a purely discontinuous, square-integrable martingale which is null at the origin (for example, see \citet[Lemma IV.21.12]{rogerswilliamsii.book}).  Note that the set of martingales $\{ \markovmartingale_{ij}; i,j \in \markovspace, i \neq j \}$ are mutually orthogonal.

\section{The control problem} \label{SECproblem}

Suppose for some $P \in \naturalnumbers$ we are given a set $\controlset \in \realnumbersP$ and a \emph{control process} $\vcontrol(t) = \vcontrol (\omega,t) : \Omega \times \timeset \rightarrow \controlset$.  We assume that the control $\vcontrol(t)$ is $\{ \filtration_{t} \}$-adapted and c\`{a}dl\`{a}g.  Consider the \emph{state variable} $\state (t) = (\state_{1} (t), \ldots, \state_{\BMdim} (t))^{\top}$  whose $n$th component satisfies the stochastic differential equation
\begin{equation} \label{EQNcontroleqn}
\ud \state_{n} (t) = \meanret_{n} (t, \state(t), \vcontrol(t), \markovchain(t_{-})) \ud t + \sum_{m=1}^{\BMdim} \varianceret_{nm} (t, \state(t), \control(t), \markovchain(t_{-})) \ud \BM_{m} (t),
\end{equation}
where $\meanret_{n}: \timeset \times \realnumbersN \times \realnumbersP \times \markovspace \rightarrow \realnumbers$ and $\varianceret_{nm}: \timeset \times \realnumbersN \times \realnumbersP \times \markovspace \rightarrow \realnumbers$ are given continuous functions for $n,m = 1, \ldots, \BMdim$.  Using $A^{\top}$ to denote the transpose of a matrix $A$, set $\meanret (t) := (\meanret_{1} (t), \ldots, \meanret_{\BMdim} (t))^{\top}$ and $\varianceret(t) := (\varianceret_{nm} (t))_{n,m=1}^{\BMdim}$.

We consider a performance criterion defined for each $\statezero \in \realnumbersN$ as
\begin{displaymath}
\Jfn^{(\vcontrol)} (\statezero) := \Jfn^{(\vcontrol)} (\statezero, \markovinitialstate) := \expectation \left( \int_{0}^{\timehorizon} \ffn (t, \state(t), \vcontrol(t), \markovchain(t)) \ud t + \gfn ( \state(\timehorizon), \markovchain(\timehorizon) ) \, \bigg\vert \, \state(0) = \statezero, \markovchain(0) = \markovinitialstate \right),
\end{displaymath}
where for each $i \in \markovspace$ we have that $\ffn(\cdot,\cdot,\cdot,i):\timeset \times \realnumbersN \times \controlset \rightarrow \realnumbers$ is continuous and $\gfn (\cdot, i):\realnumbersN \rightarrow \realnumbers$ is $C^{1} (\realnumbers)$ and concave.

We say that the control process $\vcontrol$ is admissible and write $\vcontrol \in \controladmiss$ if, for each $\statezero \in \realnumbersN$, (\ref{EQNcontroleqn}) has a unique, strong solution $\state(t) = \state^{(\vcontrol)} (t)$, $t \in \timeset$ satisfying both $\state(0) = \statezero$, a.s., and
\begin{displaymath}
\expectation \left( \int_{0}^{\timehorizon} \ffn (t, \state(t), \vcontrol(t), \markovchain(t)) \ud t + \gfn ( \state(\timehorizon), \markovchain(\timehorizon) ) \right) < \infty.
\end{displaymath}
The stochastic control problem is to find an optimal control $\vcontrolopt \in \controladmiss$ such that
\begin{equation} \label{EQNprobforSMP}
\Jfn^{(\vcontrolopt)} (\statezero) = \sup_{\vcontrol \in \controladmiss} \Jfn^{(\vcontrol)} (\statezero).
\end{equation}
Define the Hamiltonian $\hamiltonian : \timeset \times \realnumbersN \times \controlset \times \markovspace \times \realnumbersN \times \realnumbersNN \rightarrow \realnumbers$ by
\begin{equation} \label{EQNhamiltonian}
\hamiltonian (t,\statezero,\vcontrol,i,\vhamp,\vhamq) := \ffn(t,\statezero,\vcontrol,i) + \meanret^{\top} (t,\statezero,\vcontrol,i) \vhamp + \textrm{tr}(\varianceretT (t,\statezero,\vcontrol,i) \vhamq),
\end{equation}
where $\textrm{tr}(A)$ denotes the trace of the matrix $A$.  We assume that the Hamiltonian $\hamiltonian$ is differentiable with respect to $\statezero$.

The adjoint equation corresponding to $\vcontrol$ and $\state^{(\vcontrol)}$ in the unknown, adapted processes $\vhamp (t) \in \realnumbersN$, $\vhamq (t) \in \realnumbersNN$ and $\vhamr (t) = (\hamr^{(1)} (t), \ldots, \hamr^{(\BMdim)} (t) )^{\top}$, where $\hamr^{(n)} \in \realnumbersDD$ for $n = 1, \ldots, \BMdim$, is the backward stochastic differential equation
\begin{equation} \label{EQNadjoint}
\left\{ \begin{array}{ll}
 \ud \vhamp(t) & = - \nabla_{x} \hamiltonian (t,\state(t),\vcontrol(t),\markovchain(t),\vhamp(t),\vhamq(t)) \ud t + \vhamqT(t) \ud \vBM (t) + \hamr (t) \bullet \ud \markovmartingale (t) \\
\vhamp(\timehorizon) & = \nabla_{x} \gfn ( \state(\timehorizon), \markovchain(\timehorizon) ), \quad \textrm{a.s.}
\end{array} \right.
\end{equation}
where $\nabla_{x} \hamiltonian (t,\state(t),\vcontrol(t),\markovchain(t),\vhamp(t),\vhamq(t))$ denotes $\nabla_{x} \hamiltonian (t,x,\vcontrol(t),\markovchain(t),\vhamp(t),\vhamq(t)) \vert_{x=\state(t)}$, \\ $\nabla_{x} \gfn ( \state(\timehorizon), \markovchain(\timehorizon) )$ denotes $\nabla_{x} \gfn ( x, \markovchain(\timehorizon) ) \vert_{x=\state(\timehorizon)}$ and, for notational convenience, we define
\begin{displaymath}
\hamr (t) \bullet \ud \markovmartingale (t) := \left( \sum_{j \neq i} \hamr_{ij}^{(1)} (t) \ud \markovmartingale_{ij} (t), \cdots, \sum_{j \neq i} \hamr_{ij}^{(\BMdim)} (t) \ud \markovmartingale_{ij} (t) \right)^{\top},
\end{displaymath}
 for all $t \in [0,\timehorizon)$.  Note that we use throughout this paper $\sum_{j \neq i}$ as shorthand for $\sum_{i=1}^{\markovdim} \sum_{\substack{j=1, \\ j \neq i}}^{\markovdim}$.

\begin{rmk}
Notice that there are jumps in the adjoint equation (\ref{EQNadjoint}) even though there are no jumps in the equation (\ref{EQNcontroleqn}) which governs the state variable $\state (t)$.  This is a consequence of the coefficients $\meanret (t)$ and $\varianceret (t)$ being functions of the Markov chain $\markovchain (t)$.  Moreover, the unknown process $\vhamr (t)$ in the adjoint equations (\ref{EQNadjoint}) does not appear in the Hamiltonian (\ref{EQNhamiltonian}).
\end{rmk}

%
%

\section{Sufficient stochastic maximum principle} \label{SECsmp}

Here we state and prove the sufficient stochastic maximum principle.  In Section \ref{SECexample}, we apply it to a \QLM problem.
\begin{thm}[Sufficient stochastic maximum principle] \label{THMstocmax}
 Let $\vcontrolhat \in \controladmiss$ with corresponding solution $\statehat = \state^{(\vcontrolhat)}$ and suppose that there exists a solution $(\vhamhatp (t),\vhamhatq (t), \vhamhatr (t))$ of the corresponding adjoint equation (\ref{EQNadjoint}) satisfying
\begin{equation} \label{EQNsmpfiniteone}
\expectation \int_{0}^{\timehorizon} \bigg\lVert \left( \varianceret (t,\statehat (t)) - \varianceret (t,\state^{(\control)} (t)) \right)^{\top} \vhamhatp (t) \bigg\rVert^{2} \ud t < \infty,
\end{equation}
\begin{equation} \label{EQNsmpfinitetwo}
 \expectation \int_{0}^{\timehorizon} \bigg\lVert \vhamhatq^{\top} (t) \left( \statehat (t) - \state^{(\control)} (t) \right)  \bigg\rVert^{2} \ud t < \infty,
\end{equation}
and
\begin{equation} \label{EQNsmpfinitethree}
\sum_{n=1}^{\BMdim} \sum_{j \neq i} \expectation \int_{0}^{\timehorizon}  \bigg\lvert \left( \statehat_{n} (t) - \state^{(\vcontrol)}_{n} (t) \right) \hamhatr_{ij}^{(n)} (t) \bigg\rvert^{2} \ud \markovangleqvprocessij (t) < \infty,
\end{equation}
for all admissible controls $\control \in \controladmiss$.  Further suppose that
\begin{enumerate}
 \item $\hamiltonian (t,\statehat(t),\vcontrolhat(t),\markovchain(t),\vhamhatp(t),\vhamhatq(t)) = \sup_{v \in \controlset} \hamiltonian (t,\statehat(t),v,\markovchain(t),\vhamhatp(t),\vhamhatq(t)), \quad \forall t \in \timeset$,
\item $\gfn(\statezero, i)$ is a concave function of $\statezero$ for each $i \in \markovspace$, and
\item for each fixed pair $(t,i) \in \timeset \times \markovspace$, $\hamiltonianhat (\statezero) := \max_{v \in \controlset} \hamiltonian (t,\statezero,v,i,\vhamhatp(t),\vhamhatq(t))$ exists and is a concave function of $\statezero$. 
\end{enumerate}
Then $\vcontrolhat$ is an optimal control.
\end{thm}

\begin{proof}
Fix $\vcontrol \in \controladmiss$ with corresponding solution $\state = \state^{(\vcontrol)}$.  For notational ease, denote the quadruple $(t,\statehat(t_{-}),\vcontrolhat(t_{-}),\markovchain(t_{-}))$ by $(t,\statehat(t_{-}))$ and similarly denote the quadruple $(t,\state(t_{-}),\vcontrol(t_{-}),\markovchain(t_{-}))$ by $(t,\state(t_{-}))$.  Then
\begin{displaymath}
 \Jfn(\vcontrolhat) - \Jfn(\vcontrol) = \expectation \left( \int_{0}^{\timehorizon} \left( \ffn(t,\statehat(t)) - \ffn(t,\state(t)) \right) \ud t + \gfn( \statehat(\timehorizon) , \markovchain(\timehorizon) ) - \gfn( \state(\timehorizon), \markovchain(\timehorizon) ) \right).
\end{displaymath}
We use the concavity of $\gfn(\cdot, i)$ for each $i \in \markovspace$ and (\ref{EQNadjoint}) to obtain the inequalities
\begin{displaymath}
\begin{split}
  \expectation \left( \gfn( \statehat(\timehorizon), \markovchain(\timehorizon) ) - \gfn( \state(\timehorizon), \markovchain(\timehorizon) ) \right) & \geq \expectation \left( \left( \statehat(\timehorizon) - \state(\timehorizon) \right)^{\top} \nabla_{x} \gfn \left( \statehat(\timehorizon), \markovchain(\timehorizon) \right) \right) \\
& \geq \expectation \left( \left( \statehat(\timehorizon) - \state(\timehorizon) \right)^{\top} \vhamhatp(\timehorizon) \right).
\end{split}
\end{displaymath}
This gives
\begin{equation} \label{EQNJoptone}
\Jfn(\vcontrolhat) - \Jfn(\vcontrol) \geq \expectation \int_{0}^{\timehorizon} \left( \ffn(t, \statehat(t_{-})) - \ffn(t, \state(t_{-})) \right) \ud t + \expectation \left( \left( \statehat(\timehorizon) - \state(\timehorizon) \right)^{\top} \vhamhatp(\timehorizon) \right).
\end{equation}
To expand the first term on the right-hand side of (\ref{EQNJoptone}), we use the definition of $\hamiltonian$ in (\ref{EQNhamiltonian}) to obtain
\begin{equation} \label{EQNJopttwo}
\begin{split}
& \expectation \int_{0}^{\timehorizon} \left( \ffn(t, \statehat(t)) - \ffn(t, \state(t)) \right) \ud t \\
& = \expectation \int_{0}^{\timehorizon} \left( \hamiltonian (t,\statehat(t),\vcontrolhat(t),\markovchain(t),\vhamhatp(t),\vhamhatq(t)) - \hamiltonian (t,\state(t),\vcontrol(t),\markovchain(t),\vhamhatp(t),\vhamhatq(t)) \right) \ud t \\
& - \expectation \int_{0}^{\timehorizon} \left( \left( \meanret (t, \statehat(t)) - \meanret (t, \state(t)) \right)^{\top} \vhamhatp(t) + \textrm{tr} \left( \varianceret (t, \statehat(t)) - \varianceret (t, \state(t)) \right)^{\top} \vhamhatq(t) \right) \ud t.
\end{split}
\end{equation}
To expand the second term on the right-hand side of (\ref{EQNJoptone}) we begin by applying integration-by-parts to get
\begin{displaymath}
 \left( \statehat(\timehorizon) - \state(\timehorizon) \right)^{\top} \vhamhatp(\timehorizon) = \int_{0}^{\timehorizon} \left( \statehat (t) - \state (t) \right)^{\top} \ud \vhamhatp (t) + \int_{0}^{\timehorizon} \vhamhatp^{\top} (t) \ud \left( \statehat (t) - \state (t) \right) + \left[ \statehat - \state, \vhamhatp \right] (\timehorizon).
\end{displaymath}
Substitute for $\state$, $\statehat$ and $\vhamhatp$ from (\ref{EQNcontroleqn}) and (\ref{EQNadjoint}) to find
\begin{displaymath}
\begin{split}
& \left( \statehat(\timehorizon) - \state(\timehorizon) \right)^{\top} \vhamhatp(\timehorizon) \\
& = \int_{0}^{\timehorizon} \left( \statehat (t) - \state (t) \right)^{\top} \left( - \nabla_{x} \hamiltonian (t,\statehat(t),\vcontrolhat(t),\markovchain(t),\vhamhatp(t),\vhamhatq(t)) \ud t + \vhamhatqT(t) \ud \vBM (t) + \hamhatr (t) \bullet \ud \markovmartingale (t) \right) \\
& + \int_{0}^{\timehorizon} \vhamhatp^{\top} (t) \left( \left( \meanret (t, \statehat(t)) - \meanret (t, \state(t)) \right) \ud t +  \left( \varianceret (t, \statehat(t)) - \varianceret (t, \state(t)) \right)^{\top} \ud \vBM (t) \right) \\
& + \int_{0}^{\timehorizon} \textrm{tr} \left( \vhamhatqT (t) \left( \varianceret (t, \statehat(t)) - \varianceret (t, \state(t)) \right) \right) \ud t.
\end{split}
\end{displaymath}
Due to the integrability conditions (\ref{EQNsmpfiniteone})-(\ref{EQNsmpfinitethree}), the Brownian motion and Markov chain martingale integrals in the latter equation are square-integrable martingales which are null at the origin.  Thus taking expectations we obtain
\begin{displaymath}
\begin{split}
 & \expectation \left( \left( \statehat(\timehorizon) - \state(\timehorizon) \right)^{\top} \vhamhatp(\timehorizon) \right) \\
& = \expectation \int_{0}^{\timehorizon} \left( - \left( \statehat (t) - \state (t) \right)^{\top} \nabla_{x} \hamiltonian(t,\statehat(t),\vcontrolhat(t),\markovchain(t),\vhamhatp(t),\vhamhatq(t)) \right) \ud t \\
& + \expectation \int_{0}^{\timehorizon} \left( \vhamhatp^{\top} (t) \left( \meanret (t, \statehat(t)) - \meanret (t, \state(t)) \right) + \textrm{tr} \left( \vhamhatqT (t) \left( \varianceret (t, \statehat(t)) - \varianceret (t, \state(t)) \right) \right) \right) \ud t.
\end{split}
\end{displaymath}
Substitute the last equation and (\ref{EQNJopttwo}) into the inequality (\ref{EQNJoptone}) to find after cancellation that
\begin{equation} \label{EQNinequJ}
\begin{split}
\Jfn(\vcontrolhat) - \Jfn(\vcontrol) & \geq \expectation \int_{0}^{\timehorizon} \bigg( \hamiltonian (t,\statehat(t),\vcontrolhat(t),\markovchain(t),\vhamhatp(t),\vhamhatq(t)) - \hamiltonian (t,\state(t),\vcontrol(t),\markovchain(t),\vhamhatp(t),\vhamhatq(t)) \\
& - \left( \statehat (t) - \state (t) \right)^{\top} \nabla_{x} \hamiltonian (t,\statehat(t),\vcontrolhat(t),\markovchain(t),\vhamhatp(t),\vhamhatq(t)) \bigg) \ud t.
\end{split}
\end{equation}
We can show that the integrand on the right-hand side of (\ref{EQNinequJ}) is non-negative a.s. for each $t \in \timeset$ by fixing the state of the Markov chain and then using the assumed concavity of $\hamiltonianhat (\statezero)$ to apply the argument of \citet[pages 83-84]{framstadetal04.article}.  This gives $\Jfn(\vcontrolhat) - \Jfn(\vcontrol) \geq 0$ and hence $\vcontrolhat$ is optimal.
\end{proof}

%
%

\section{Connection to Dynamic Programming} \label{SECdynamic}
In a jump-diffusion setting, the connection between the stochastic maximum principle and dynamic programming principle is shown in \citet[Section 3]{framstadetal04.article}.  We show a similar connection in Theorem \ref{THMreltoDP}, between the value function $V(t,x,i)$ of the control problem and the adjoint processes $\hamp (t)$, $\hamq (t)$ and $\hamr (t)$.  The main difference is that, in the regime-switching diffusion model, the adjoint process $\vhamr_{ij} (t)$ represents the jumps of the $x$-gradient of the value function due to the Markov chain switching from state $i$ to state $j$.  In the non-regime-switching jump-diffusion model, this adjoint process represents the jumps of the $x$-gradient of the value function due to the jumps in the state process $X(t)$.

To put the problem in a Markovian framework so that we can apply dynamic programming, define
\begin{displaymath}
 \Jfncontrol (s,x,i) := \expectation \left( \int_{s}^{\timehorizon} \ffn \left( t, X (t), \control(t), \markovchain(t) \right) \ud t + \gfn( X (\timehorizon), \markovchain(\timehorizon) ) \, \bigg\vert X(s)=x, \markovchain(s) = i \right), \quad \forall \vcontrol \in \controladmiss,
\end{displaymath}
and put
\begin{equation} \label{EQNprobforDP}
  \valuefn (s,x,i) := \sup_{\vcontrol \in \controladmiss} \Jfncontrol (s,x,i),
\end{equation}
for all $(s,x,i) \in \timeset \times \realnumbersN \times \markovspace$.
\begin{thm} \label{THMreltoDP}
 Assume that $V(\cdot,\cdot,i) \in C^{1,3}(\timeset \times \realnumbersN)$ for each $i \in \markovspace$ and that there exists an optimal Markov control $\vcontrolopt(t,x,i)$ for (\ref{EQNprobforDP}), with corresponding solution $\stateopt = \state^{(\vcontrolopt)}$.  Define
\begin{equation} \label{EQNdphamp}
\hamp_{n}(t) := \frac{\partial \valuefn}{\partial x_{n}} (t, \stateopt(t), \markovchain(t)),
\end{equation}
\begin{equation} \label{EQNdphamq}
\hamq_{nm} (t) := \sum_{l=1}^{\BMdim} \varianceret_{lm} (t, \stateopt(t), \vcontrolopt(t), \markovchain(t)) \, \frac{\partial^{2} \valuefn}{\partial x_{n} \partial x_{l}} (t, \stateopt(t), \markovchain(t)),
\end{equation}
\begin{equation} \label{EQNdphamr}
\vhamr_{ij}^{(n)} (t) := \frac{\partial \valuefn}{\partial x_{n} } (t, \stateopt(t), j) - \frac{\partial \valuefn}{\partial x_{n}} (t, \stateopt(t), i).
\end{equation}
Then $\vhamp(t)$, $\vhamq(t)$ and $\vhamr(t)$ solve the adjoint equation (\ref{EQNadjoint}).
\end{thm}

\begin{rmk}
To prove the above theorem, we require It{\^o}'s formula, which is given next.  It{\^o}'s formula can be found in \citet[Theorem 18, page 278]{protter.book}.
\end{rmk}

\begin{thm}[It{\^o}'s formula] \label{THMitoformula}
Suppose we are given an $\BMdim$-dimensional process $\ITOX = (\ITOX_{1}, \ldots, \ITOX_{\BMdim} )^{\top}$ satisfying for each $n=1, \ldots, \BMdim$
\begin{displaymath}
\begin{split}
 \ud \ITOX_{n} (t) & = \meanret_{n} (t, \vITOX (t), \markovchain(t_{-})) \ud t + \sum_{m=1}^{\BMdim} \varianceret_{nm} (t, \vITOX (t), \markovchain(t_{-})) \ud \BM_{m} (t) \\
\ITOX_{n} (0) & = x_{0}^{(n)}, \textrm{ a.s.},
\end{split}
\end{displaymath}
for some $x_{0}^{(n)} \in \realnumbers$, and functions $V(\cdot,\cdot,i) \in C^{1,3}(\timeset \times \realnumbersN)$ for each $i=1,\ldots,\markovdim$.  Then
\begin{displaymath}
\begin{split}
 V(t,\vITOX (t),\markovchain (t)) & = V(0,\vITOX(0),\markovchain(0)) + \int_{0}^{t} \Gamma V(s,\vITOX (s),\markovchain(s_{-})) \ud s \\
& + \sum_{n=1}^{\BMdim} \int_{0}^{t} \frac{\partial V}{\partial x_{n}} (s,\vITOX (s),\markovchain(s_{-})) \sum_{m=1}^{\BMdim} \varianceret_{nm} (s,\vITOX (s),\markovchain(s_{-})) \ud \BM_{m} (s) \\
& + \sum_{j \neq i} \int_{0}^{t} \left( V(s,\vITOX (s),j) - V(s,\vITOX (s),i) \right) \ud \markovmartingale_{ij} (t),
\end{split}
\end{displaymath}
for
\begin{displaymath}
\begin{split}
\Gamma V(t,x,i) & := \frac{\partial V}{\partial t} (t,x,i) + \sum_{n=1}^{\BMdim} \frac{\partial V}{\partial x_{n}} (t,x,i) \meanret_{n} (t,x,i) \\
& + \frac{1}{2} \sum_{n=1}^{\BMdim} \sum_{m=1}^{\BMdim} \frac{\partial^{2} V}{\partial x_{n} \partial x_{m}} (t,x,i) \sum_{l=1}^{\BMdim} \varianceret_{nl} (t,x,i) \varianceret_{ml} (t,x,i) \\
& + \sum_{j=1}^{\markovdim} \markovgenerator_{ij} \left( V(t, x, j) - V(t, x, i) \right),
\end{split}
\end{displaymath}
for all $(t,x,i) \in \timeset \times \realnumbersN \times \markovspace$.
\end{thm}

\begin{proof}[Proof of Theorem \ref{THMreltoDP}]
From general dynamic programming theory, the Hamilton-Jacobi-Bellman equation holds:
\begin{displaymath}
\frac{\partial \valuefn}{\partial t} (t,x,i) + \sup_{\vcontrol \in \controlset} \left\{ \ffn (t,x,u,i) + \mathcal{A}^{u} \valuefn(t, x, i) \right\} = 0,
\end{displaymath}
where $\mathcal{A}^{u}$ is the infinitesimal generator and the supremum is attained by $\vcontrolopt(t,x,i)$.  Define
\begin{displaymath}
 \HJBF (t,x,u,i) := \frac{\partial \valuefn}{\partial t} (t,x,i) + \ffn (t,x,u,i) + \mathcal{A}^{u} \valuefn(t, x, i).
\end{displaymath}
Using It{\^o}'s formula (Theorem \ref{THMitoformula}) to expand $\mathcal{A}^{u} \valuefn(t, x, i)$, we find
\begin{displaymath}
\begin{split}
 \HJBF (t,x,u,i) & = \ffn (t,x,u,i) + \frac{\partial \valuefn}{\partial t} (t,x,i) + \sum_{n=1}^{\BMdim} \frac{\partial \valuefn}{\partial x_{n}} (t,x,i) \meanret_{n} (t,x,i) \\
& + \frac{1}{2} \sum_{n=1}^{\BMdim} \sum_{m=1}^{\BMdim} \frac{\partial^{2} \valuefn}{\partial x_{n} \partial x_{m}} (t,x,i) \sum_{l=1}^{\BMdim} \varianceret_{nl} (t,x,i) \varianceret_{ml} (t,x,i) \\
& + \sum_{j=1}^{\markovdim} \markovgenerator_{ij} \left( \valuefn(t, x, j) - \valuefn(t, x, i) \right).
\end{split}
\end{displaymath}
Differentiate $\HJBF (t,x,\vcontrolopt(t,x,i),i)$ with respect to $x_{\kwrtx}$ and evaluate at $x=\stateopt(t)$ and $i=\markovchain (t)$.  For notational ease denote the quadruple $(t,\stateopt(t),\vcontrolopt(t,\stateopt(t),\markovchain (t)),\markovchain (t))$ by $(t,\markovchain (t))$.  We get
\begin{equation} \label{EQNHJBFoptsub}
\begin{split}
 0 & = \frac{\partial \ffn}{\partial x_{\kwrtx}} (t,\markovchain (t)) + \frac{\partial^{2} \valuefn}{\partial x_{\kwrtx} \partial t} (t,\stateopt(t),\markovchain (t)) + \sum_{n=1}^{\BMdim} \frac{\partial^{2} V}{\partial x_{\kwrtx} \partial x_{n}} (t,\stateopt(t),\markovchain (t)) \cdot \meanret_{n} (t,\markovchain (t)) \\
& + \sum_{n=1}^{\BMdim} \frac{\partial V}{\partial x_{n}} (t,\stateopt(t),\markovchain (t)) \cdot \frac{\partial \meanret_{n}}{\partial x_{\kwrtx}} (t,\markovchain (t)) \\
& + \frac{1}{2} \sum_{n=1}^{\BMdim} \sum_{m=1}^{\BMdim} \frac{\partial^{3} V}{\partial x_{\kwrtx} \partial x_{n} \partial x_{m}} (t,\stateopt(t),\markovchain (t)) \left( \sum_{l=1}^{\BMdim} \varianceret_{nl} \varianceret_{ml} \right) (t,\markovchain (t)) \\
& + \frac{1}{2} \sum_{n=1}^{\BMdim} \sum_{m=1}^{\BMdim} \frac{\partial^{2} V}{\partial x_{n} \partial x_{m}} (t,\stateopt(t),\markovchain (t)) \frac{\partial}{\partial x_{\kwrtx}} \left( \sum_{l=1}^{\BMdim} \varianceret_{nl} \varianceret_{ml} \right) (t,\markovchain (t)) \\
& + \sum_{j=1}^{\markovdim} \markovgenerator_{\markovchain (t), j} \left( \frac{\partial \valuefn}{\partial x_{\kwrtx}} (t,\stateopt(t),j) - \frac{\partial \valuefn}{\partial x_{\kwrtx}} (t,\stateopt(t),\markovchain (t)) \right).
\end{split}
\end{equation}
Next define
\begin{displaymath}
 \fnYk (t) := \frac{\partial \valuefn}{\partial x_{\kwrtx}} (t, \stateopt(t), \markovchain (t)), \quad \textrm{for $\kwrtx=1, \ldots, \BMdim$}.
\end{displaymath}
Using It{\^o}'s formula (Theorem \ref{THMitoformula}) to obtain the dynamics of  $\fnYk (t)$, we find
\begin{displaymath}
\begin{split}
  \ud \fnYk (t) & = \bigg\{ \frac{\partial^{2} \valuefn}{\partial t \partial x_{\kwrtx}} (t, \stateopt(t), \markovchain (t)) + \sum_{n=1}^{\BMdim} \frac{\partial^{2} V}{\partial x_{n} \partial x_{\kwrtx}} (t, \stateopt(t), \markovchain (t)) \cdot \meanret_{n} (t, \markovchain (t)) \\
& + \frac{1}{2} \sum_{n=1}^{\BMdim} \sum_{m=1}^{\BMdim} \frac{\partial^{3} V}{\partial x_{n} \partial x_{m} \partial x_{\kwrtx}} (t,\stateopt(t),\markovchain (t)) \left( \sum_{l=1}^{\BMdim} \varianceret_{nl} \varianceret_{ml} \right) (t,\markovchain (t)) \\
& + \sum_{j=1}^{\markovdim} \markovgenerator_{\markovchain (t),j} \left( \frac{\partial \valuefn}{\partial x_{\kwrtx}} (t,\stateopt(t),j) - \frac{\partial \valuefn}{\partial x_{\kwrtx}} (t,\stateopt(t),\markovchain (t)) \right) \bigg\} \ud t \\
& + \sum_{n=1}^{\BMdim} \frac{\partial^{2} V}{\partial x_{n} \partial x_{\kwrtx}} (t, \stateopt(t), \markovchain (t)) \sum_{m=1}^{\BMdim} \varianceret_{nm} (t,\markovchain (t)) \ud \BM_{m} (t) \\
& + \sum_{j \neq i} \left( \frac{\partial \valuefn}{\partial x_{\kwrtx}} (t,\stateopt(t),j) - \frac{\partial \valuefn}{\partial x_{\kwrtx}} (t,\stateopt(t),i) \right) \ud \markovmartingale_{ij} (t).
\end{split}
\end{displaymath}
Substituting for $\frac{\partial^{2} \valuefn}{\partial t \partial x_{\kwrtx}}$ from (\ref{EQNHJBFoptsub}), we get
\begin{equation} \label{EQNdpfnYkfinal}
\begin{split}
  \ud \fnYk (t) = & - \bigg\{ \frac{\partial \ffn}{\partial x_{\kwrtx}} (t,\markovchain (t)) + \sum_{n=1}^{\BMdim} \frac{\partial V}{\partial x_{n}} (t,\stateopt(t),\markovchain (t)) \cdot \frac{\partial \meanret_{n}}{\partial x_{\kwrtx}} (t,\markovchain (t)) \\
+ & \frac{1}{2} \sum_{n=1}^{\BMdim} \sum_{m=1}^{\BMdim} \frac{\partial^{2} V}{\partial x_{n} \partial x_{m}} (t,\stateopt(t),\markovchain (t)) \frac{\partial}{\partial x_{\kwrtx}} \left( \sum_{l=1}^{\BMdim} \varianceret_{nl} \varianceret_{ml} \right) (t,\markovchain (t)) \bigg\} \ud t \\
+ & \sum_{n=1}^{\BMdim} \frac{\partial^{2} V}{\partial x_{n} \partial x_{\kwrtx}} (t, \stateopt(t), \markovchain (t)) \sum_{m=1}^{\BMdim} \varianceret_{nm} (t,\markovchain (t)) \ud \BM_{m} (t) \\
+ & \sum_{j \neq i} \left( \frac{\partial \valuefn}{\partial x_{\kwrtx}} (t,\stateopt(t),j) - \frac{\partial \valuefn}{\partial x_{\kwrtx}} (t,\stateopt(t),i) \right) \ud \markovmartingale_{ij} (t).
\end{split}
\end{equation}
Note that
\begin{equation} \label{EQNdpfnrearrangesigma}
\begin{split}
 \frac{1}{2} \sum_{n=1}^{\BMdim} \sum_{m=1}^{\BMdim} \frac{\partial^{2} V}{\partial x_{n} \partial x_{m}} \frac{\partial}{\partial x_{\kwrtx}} \left( \sum_{l=1}^{\BMdim} \varianceret_{nl} \varianceret_{ml} \right) = & \frac{1}{2} \sum_{n=1}^{\BMdim} \sum_{m=1}^{\BMdim} \frac{\partial^{2} V}{\partial x_{n} \partial x_{m}} \sum_{l=1}^{\BMdim} \left( \frac{\partial \varianceret_{nl} }{\partial x_{\kwrtx}}  \varianceret_{ml} + \varianceret_{nl} \frac{\partial \varianceret_{ml} }{\partial x_{\kwrtx}} \right) \\
= & \sum_{m=1}^{\BMdim} \sum_{l=1}^{\BMdim} \left( \sum_{n=1}^{\BMdim} \varianceret_{nl} \frac{\partial^{2} V}{\partial x_{n} \partial x_{m}} \right) \frac{\partial \varianceret_{ml} }{\partial x_{\kwrtx}}.
\end{split}
\end{equation}
Next, from (\ref{EQNhamiltonian}) we find that
\begin{displaymath}
\begin{split}
\frac{\partial \hamiltonian}{\partial x_{\kwrtx}} (t,\state(t),\vcontrol(t),\markovchain(t),\vhamp(t),\vhamq(t)) = & \frac{\partial \ffn}{\partial x_{\kwrtx}} (t,\markovchain (t)) + \sum_{n=1}^{\BMdim} \frac{\partial \meanret_{n}}{\partial x_{\kwrtx}} (t,\markovchain (t)) \hamp_{n}(t) \\
+ & \sum_{n=1}^{\BMdim} \sum_{m=1}^{\BMdim} \frac{\partial \varianceret_{nm}}{\partial x_{\kwrtx}} (t,\markovchain (t)) \hamq_{nm} (t).
\end{split}
\end{displaymath}
Substituting (\ref{EQNdphamp}) - (\ref{EQNdphamr}), (\ref{EQNdpfnrearrangesigma}) and the last equation into (\ref{EQNdpfnYkfinal}) gives
\begin{displaymath}
  \ud \fnYk (t) = - \frac{\partial \hamiltonian}{\partial x_{\kwrtx}} (t,\state(t),\vcontrol(t),\markovchain(t),\vhamp(t),\vhamq(t)) \ud t + \sum_{m=1}^{\BMdim} \hamq_{km} (t) \ud \BM_{m} (t) + \sum_{j \neq i} \hamr_{ij}^{(k)} (t) \ud \markovmartingale_{ij} (t),
\end{displaymath}
and as $\fnYk (t)=\hamp_{k}(t)$ for each $k=1, \ldots, \BMdim$, we have shown that $\vhamp(t)$, $\vhamq(t)$ and $\hamr (t)$ given by (\ref{EQNdphamp})-(\ref{EQNdphamr}) solve the adjoint equation (\ref{EQNadjoint}).
\end{proof}

\section{Application: \QLM problem} \label{SECexample}

We demonstrate the use of the maximum principle by solving a \QLM problem.  Consider a regime-switching financial market that is built upon one traded asset, which we call the risky asset, and a risk-free asset.  The risk-free asset's price process $\bankaccount = \{ \bankaccount (t), t \in \timeset \}$ is given by
\begin{equation} \label{MKTriskfreeprice}
 \frac{\ud \bankaccount (t)}{\bankaccount (t) } = r (t, \markovchain (t_{-})) \ud t, \quad \forall t \in \timeset, \quad \bankaccount (0) = 1,
\end{equation}
where the risk-free rate of return $r(t,i)$ is a bounded, deterministic function on $\timeset$ for $i=1, \ldots, \markovdim$.

The price process $\stockprice = \{ \stockprice (t), t \in \timeset \}$ of the risky asset is given by 
\begin{equation} \label{MKTriskyprice}
 \frac{\ud \stockprice (t)}{\stockprice (t) } = b (t, \markovchain (t_{-})) \ud t + \sigma (t, \markovchain (t_{-})) \ud \BM (t), \quad \forall t \in \timeset,
\end{equation}
with the initial value $\stockprice (0)$ being a fixed, strictly positive constant in $\realnumbers$.  We assume that the mean rate of return $b (t,i)$ and the volatility process $\sigma (t,i)$ are bounded, non-zero, deterministic functions on $\timeset$ for $i=1, \ldots, \markovdim$.  Here, $\BM$ is a 1-dimensional standard Brownian motion and $b$ and $\sigma$ are scalar processes.

A portfolio process $\pi (t)$ is a $\{ \filtration_{t} \}$-previsible scalar process which gives the amount invested in the risky asset at time $t$.  Denote by $\pi_{0} (t)$ the amount invested in the risk-free asset at time $t$.  The corresponding wealth process $X^{\pi} (t)$ is then given by
\begin{displaymath}
 X^{\pi} (t) = \pi_{0} (t) +  \pi (t).
\end{displaymath}
We assume that at time 0, $X^{\pi} (0) = x_{0}$, a.s.  Define the market price of diffusion risk $\theta (t,i) := \sigma^{-1} (t,i) ( b (t,i) - r(t,i))$.  Under the self-financing condition, the dynamics of the wealth process satisfy
\begin{equation}  \label{EQNstateproc}
\ud X^{\pi} (t) = \left( r(t) X^{\pi} (t) + \pi (t) \sigma (t) \theta(t) \right) \ud t + \pi (t) \sigma (t) \, \ud \BM(t), \quad X^{\pi} (0) = \initialstate.
\end{equation}
We say that $\pi (t)$ is an admissible portfolio process and write $\pi \in \admissibleportfolios$, if it is a $\{ \filtration_{t} \}$-previsible, square-integrable, scalar process.

We consider the problem of finding an admissible portfolio process $\bar{\pi} \in \admissibleportfolios$ such that
\begin{displaymath}
 \expectation \left( X^{\bar{\pi}} (\timehorizon) - \constc \right)^{2} = \inf_{\pi \in \admissibleportfolios} \expectation \left( X^{\pi} (\timehorizon) - \constc \right)^{2},
\end{displaymath}
for some fixed constant $\constc \in \realnumbers$.

To solve this, we use the sufficient maximum principle of Theorem \ref{THMstocmax}.  Define the real-valued function $\gfn (x):=- (x - \constc)^{2}$ and consider the equivalent problem of maximizing
\begin{equation} \label{EQNgtomin}
 \expectation \left( \gfn (X^{\pi} (\timehorizon)) \right) = \expectation \left( - \left( X^{\pi} (\timehorizon) - \constc \right)^{2} \right).
\end{equation}
over all $\pi \in \admissibleportfolios$.  Set the control process $\control (t) := \pi (t)$ and $X (t) := X^{\pi} (t)$.  For this example, the Hamiltonian (\ref{EQNhamiltonian}) becomes
\begin{equation} \label{EQNhamiltoniancase}
\hamiltonian (t,x,\control,i,\hamp,\hamq) := \left( r (t,i) x + \control \sigma (t,i) \theta (t,i) \right) \hamp + \control \sigma(t,i) \hamq,
\end{equation}
and the adjoint equations (\ref{EQNadjoint}) are for all $t \in [0,\timehorizon)$,
\begin{equation} \label{EQNadjointeqncase}
\left\{ \begin{array}{ll}
 \ud \hamp(t) & = - r(t) \hamp (t) \ud t + \hamq(t) \ud \BM (t) + \sum_{j \neq i} \hamr_{ij} (t) \ud \markovmartingale_{ij} (t), \\
\hamp(\timehorizon) & = - 2 X (\timehorizon) + 2 \constc, \quad \textrm{a.s.}
\end{array} \right.
\end{equation}
We seek the solution $(\hamp (t), \hamq(t), \hamr(t))$ to (\ref{EQNadjointeqncase}).  Since $\gfn(x)$ is quadratic in $x$ and the adjoint process $\hamp$ is the first derivative of the function $\gfn$, a natural assumption is that $\hamp$ is linear in $X$.  This means that $\hamp$ is of the form
\begin{equation} \label{EQNadjointguess}
 \hamp (t) = \adjpphi (t, \markovchain(t)) X (t) + \adjppsi (t, \markovchain(t)),
\end{equation}
where $\adjpphi(\cdot,i)$ and $\adjppsi(\cdot,i)$ are deterministic, differentiable functions for each $i=1,\ldots,\markovdim$, which are to be found.  From (\ref{EQNadjointeqncase}), $\adjpphi$ and $\adjppsi$ have terminal boundary conditions
\begin{equation} \label{EQNtermBCs}
 \adjpphi(\timehorizon,i) = -2 \quad \textrm{and} \quad \adjppsi(\timehorizon,i) = 2 \constc, \quad \forall i \in \markovspace.
\end{equation}
The next step is to expand the right-hand side of (\ref{EQNadjointguess}) and then compare it with (\ref{EQNadjointeqncase}).  To do this, we begin by noting from It\^{o}'s formula (Theorem \ref{THMitoformula}) that for a function $f(t,\markovchain(t))$ we have
\begin{equation} \label{EQNitomarkovexp}
\begin{split}
  \ud f(t,\markovchain(t)) & = f_{t} (t,\markovchain(t_{-})) \ud t + \sum_{j \neq i} \markovgenerator_{ij} \left( f(t,j) - f(t,i) \right) \indicatorfn [ \markovchain (t_{-}) = i ]\ud t \\
& + \sum_{j \neq i} \left( f(t,j) - f(t,i) \right) \ud \markovmartingale_{ij} (t).
\end{split}
\end{equation}
Using (\ref{EQNitomarkovexp}) to expand the functions $\adjpphi$ and $\adjppsi$, and (\ref{EQNstateproc}) to expand $X$ (with $\pi (t):= \control (t)$ and $X^{\pi} (t) := X(t)$), we apply integration-by-parts to (\ref{EQNadjointguess}) to get
\begin{displaymath}
\begin{split}
 \ud \vhamp (t) = \sum_{i=1}^{\markovdim} & \indicatorfn [ \markovchain (t_{-}) = i ] \\
& \bigg\{ X (t_{-}) \left( \adjpphi(t,i) r(t,i) + \adjpphi_{t} (t,i) + \sum_{j=1}^{\markovdim} \markovgenerator_{ij} \left( \adjpphi(t,j) - \adjpphi(t,i) \right) \right) \\
& + \adjpphi(t,i) \control (t) \sigma (t,i) \theta (t,i) + \adjppsi_{t}(t,i) + \sum_{j=1}^{\markovdim} \markovgenerator_{ij} \left( \adjppsi(t,j) - \adjppsi(t,i) \right) \bigg\} \ud t \\
 & + \adjpphi(t) \control (t) \sigma (t) \ud \BM (t) \\
& + \sum_{j \neq i} \bigg( X (t_{-}) \left( \adjpphi(t,j) - \adjpphi(t,i) \right) + \left( \adjppsi(t,j) - \adjppsi(t,i) \right) \bigg) \ud \markovmartingale_{ij} (t) 
\end{split}
\end{displaymath}
Comparing coefficients with (\ref{EQNadjointeqncase}), we obtain three equations
\begin{equation} \label{EQNdtcoeffs}
\begin{split}
& - r(t,\markovchain(t_{-})) \vhamp (t_{-}) \\
& = \sum_{i=1}^{\markovdim} \indicatorfn [ \markovchain (t_{-}) = i ] \bigg\{ X (t_{-}) \left( \adjpphi(t,i) r(t,i) + \adjpphi_{t} (t,i) + \sum_{j=1}^{\markovdim} \markovgenerator_{ij} \left( \adjpphi(t,j) - \adjpphi(t,i) \right) \right) \\
& \qquad + \adjpphi(t,i) \control (t) \sigma (t,i) \theta (t,i) + \adjppsi_{t}(t,i) + \sum_{j=1}^{\markovdim} \markovgenerator_{ij} \left( \adjppsi(t,j) - \adjppsi(t,i) \right) \bigg\},
\end{split}
\end{equation}
\begin{equation} \label{EQNdWcoeffs}
\vhamq(t) = \adjpphi(t) \sigma (t) \control (t),
\end{equation}
\begin{equation} \label{EQNdMcoeffs}
\hamr_{ij} (t) = X (t_{-}) \left( \adjpphi(t,j) - \adjpphi(t,i) \right) + \left( \adjppsi(t,j) - \adjppsi(t,i) \right).
\end{equation}
Let $\hat{\control} \in \admissibleportfolios$ be a candidate for the optimal control with corresponding state process $\hat{X}$ and adjoint solution $(\hamhatp, \hamhatq, \hamhatr)$.  Then for the Hamiltonian (\ref{EQNhamiltoniancase}), for all $\control \in \realnumbers$,
\begin{displaymath}
\hamiltonian (t,\hat{X}(t),\control,\alpha(t),\hamhatp(t),\hamhatq(t)) = \left( r (t) \hat{X}(t) + \control \sigma (t) \theta (t) \right) \hamhatp(t) + \control \sigma(t) \hamhatq(t).
\end{displaymath}
As this is a linear function of $\control$, we guess that the coefficient of $\control$ vanishes at optimality, which results in the equality
\begin{equation} \label{EQNhamhatqsoln}
\hamhatq(t) = - \theta (t) \hamhatp (t).
\end{equation}
Substituting into (\ref{EQNdWcoeffs}) for $\hamhatq (t)$ from (\ref{EQNhamhatqsoln}) and using (\ref{EQNadjointguess}) to replace $\hamhatp (t)$, we get
\begin{equation} \label{EQNcontrolsol}
 \hat{\control} (t) = - \sigma^{-1} (t) \theta (t) \left( \hat{X}(t) + \adjpphi^{-1} (t) \adjppsi(t) \right)
\end{equation}
 Therefore, to find the optimal control it remains to find $\adjpphi$  and $\adjppsi$.  To do this, we set $X(t) := \hat{X}(t)$, $\control(t) := \hat{\control} (t)$ and $\vhamp (t) := \hamhatp (t)$ in (\ref{EQNdtcoeffs}) and then substitute for $\hamhatp (t)$ from (\ref{EQNadjointguess}) and for $\hat{\control} (t)$ from (\ref{EQNcontrolsol}).  This results in a linear equation in $\hat{X}(t)$.  Assuming that the coefficient of $\hat{X}(t)$ equals zero, we obtain two equations
\begin{equation} \label{EQNeqnsolnadjphi}
 \adjpphi(t,i) \left( 2 r(t,i) - \lvert \theta (t,i) \rvert^{2} \right) + \adjpphi_{t}(t,i) + \sum_{j=1}^{\markovdim} \markovgenerator_{ij} \left( \adjpphi(t,j) - \adjpphi(t,i) \right) = 0,
\end{equation}
\begin{equation} \label{EQNeqnsolnadjppsi}
 \adjppsi(t,i) \left( r(t,i) - \lvert \theta (t,i) \rvert^{2} \right) + \adjppsi_{t}(t,i) + \sum_{j=1}^{\markovdim} \markovgenerator_{ij} \left( \adjppsi(t,j) - \adjppsi(t,i) \right) = 0,
\end{equation}
with terminal boundary conditions given by (\ref{EQNtermBCs}).  Consider the processes
\begin{equation} \label{EQNeqnsolnadj}
  \tilde{\adjpphi}(t, \markovchain (t)) := -2 \, \expectation \left( \exp \bigg\{ \int_{t}^{\timehorizon} \left( 2 r (s) - \lvert \theta (s) \rvert^{2} \right) \ud s \bigg\} \,\bigg\vert \, \markovchain(t) \right)
\end{equation}
and 
\begin{equation} \label{EQNeqnsolnadjii}
\tilde{\adjppsi} (t, \markovchain (t)) := 2 \constc \, \expectation \left(  \exp \bigg\{ \int_{t}^{\timehorizon} \left( r (s) - \lvert \theta (s) \rvert^{2} \right) \ud s \bigg\} \, \bigg\vert \, \markovchain(t) \right).
\end{equation}
We aim to show that $\adjpphi = \tilde{\adjpphi}$ and $\adjppsi = \tilde{\adjppsi}$.  It is helpful to define at this point the following martingales:
\begin{equation} \label{EQNCdiscountPsi}
 \fnpsi (t) := \expectation \left( \exp \bigg\{ \int_{0}^{\timehorizon} \left( 2 r (s) - \lvert \theta (s) \rvert^{2} \right) \ud s \bigg\} \, \bigg\vert \, \filtration_{t}^{\markovchain} \right)
\end{equation}
and
\begin{equation} \label{EQNCdiscountfnP}
\fnP (t) := \expectation \left( \exp \bigg\{ \int_{0}^{\timehorizon} \left( r (s) - \lvert \theta (s) \rvert^{2} \right) \ud s \bigg\} \bigg\vert \, \filtration_{t}^{\markovchain} \right),
\end{equation}
where $\filtration_{t}^{\markovchain}:= \sigma \{ \markovchain (\tau), \tau \in [0,t] \} \vee \mathcal{N} ( \Pmeasure )$ is the filtration generated by the Markov chain.  From the $\{ \filtration_{t}^{\markovchain}\}$-martingale representation theorem, there exists $\{ \filtration_{t}^{\markovchain} \}$-previsible, square-integrable processes $\integrandpsi (t), \integrandP (t)$ such that
\begin{displaymath}
 \fnpsi (t) = \fnpsi (0) + \sum_{j \neq i} \int_{0}^{t} \integrandpsi_{ij} (\tau) \ud \markovmartingale_{ij} (\tau) \quad \textrm{and} \quad \fnP (t) = \fnP (0) + \sum_{j \neq i} \int_{0}^{t} \integrandP_{ij} (\tau) \ud \markovmartingale_{ij} (\tau).
\end{displaymath}
By the positivity of $\fnpsi (t)$ and $\fnP (t)$, we can define the processes $\integrandpsihat_{ij} (t) := \integrandpsi_{ij} (t) \fnpsi^{-1} (t_{-})$ and $\integrandPhat_{ij} (t) := \integrandP_{ij} (t) \fnP^{-1} (t_{-})$ so that
\begin{equation} \label{EQNMRTCDexp}
 \fnpsi (t) = \fnpsi (0) + \sum_{j \neq i} \int_{0}^{t} \fnpsi (\tau_{-}) \integrandpsihat_{ij} (\tau) \ud \markovmartingale_{ij} (\tau) \quad \textrm{and} \quad \fnP (t) = \fnP (0) + \sum_{j \neq i} \int_{0}^{t} \fnP (\tau_{-}) \integrandPhat_{ij} (\tau) \ud \markovmartingale_{ij} (\tau).
\end{equation}

From (\ref{EQNeqnsolnadj}) and the definition of $\fnpsi$ in (\ref{EQNCdiscountPsi}), we have the relationship
\begin{equation} \label{EQNfnpsiaso}
 \fnpsi (t) = - \frac{1}{2} \tilde{\adjpphi} (t, \markovchain (t)) \exp \bigg\{ \int_{0}^{t} \left( 2 r (s) - \lvert \theta (s) \rvert^{2} \right) \ud s \bigg\}, \qquad \forall t \in \timeset.
\end{equation}
Using the It\^{o} formula expansion of $\tilde{\adjpphi}(t, \markovchain (t))$ (see (\ref{EQNitomarkovexp})), we apply integration-by-parts to expand the right-hand side of the above equation and comparing it with the martingale representation of $\fnpsi (t)$ given by (\ref{EQNMRTCDexp}), we find that $\tilde{\adjpphi}$ satisfies (\ref{EQNeqnsolnadjphi}) with $\adjpphi := \tilde{\adjpphi}$.  We conclude that $\adjpphi = \tilde{\adjpphi}$.

Similarly, from (\ref{EQNeqnsolnadjii}) and the definition of $\fnP$ in (\ref{EQNCdiscountfnP}), we have
\begin{equation} \label{EQNfnpsiasi}
 \fnP (t) = \frac{1}{2 \constc} \tilde{\adjppsi} (t, \markovchain (t)) \exp \bigg\{ \int_{0}^{t} \left( r (s) - \lvert \theta (s) \rvert^{2} \right) \ud s \bigg\}, \qquad \forall t \in \timeset.
\end{equation}
Using the It\^{o} formula expansion of $\tilde{\adjppsi} (t, \markovchain (t))$ (see (\ref{EQNitomarkovexp})), we apply integration-by-parts to expand the right-hand side of the above equation and comparing it with $\fnP (t)$ given by (\ref{EQNMRTCDexp}), we find that $\tilde{\adjppsi}$ satisfies (\ref{EQNeqnsolnadjppsi}) with $\adjppsi := \tilde{\adjppsi}$.  We conclude that $\adjppsi = \tilde{\adjppsi}$.  Thus from (\ref{EQNadjointguess}), (\ref{EQNdWcoeffs}) and (\ref{EQNdMcoeffs}), we can write down the solutions
\begin{displaymath}
 \hamhatp (t) = \adjpphi (t) \hat{X} (t) + \adjppsi (t), \quad \hamhatq(t) = \adjpphi(t) \sigma (t) \hat{\control} (t), \quad \hamhatr_{ij} (t) = \hat{X} (t_{-}) \left( \adjpphi(t,j) - \adjpphi(t,i) \right) + \left( \adjppsi(t,j) - \adjppsi(t,i) \right).
\end{displaymath}
to the adjoint equation (\ref{EQNadjointeqncase}).  Substitute into (\ref{EQNcontrolsol}) for $\adjpphi = \tilde{\adjpphi}$ from (\ref{EQNfnpsiaso}) and for $\adjppsi = \tilde{\adjppsi}$ from (\ref{EQNfnpsiasi}) and use the Markov property of $\markovchain$ to obtain the control process
\begin{equation} \label{EQNoptimalcontrol}
 \hat{\control} (t) = - \left( \hat{X} (t) - \constc \frac{\expectation \left( \exp \bigg\{ \int_{t}^{\timehorizon} \left( r (s) - \lvert \theta (s) \rvert^{2} \right) \ud s \bigg\} \, \bigg\vert \, \markovchain(t) \right)}{\expectation \left( \exp \bigg\{ \int_{t}^{\timehorizon} \left( 2 r (s) - \lvert \theta (s) \rvert^{2} \right) \ud s \bigg\} \, \bigg\vert \, \markovchain(t) \right)} \right) \sigma^{-1} (t) \theta (t).
\end{equation}
With this choice of control process and the boundedness conditions on the market parameters $r$, $b$ and $\sigma$, the conditions of Theorem \ref{THMstocmax} are satisfied and hence $\hat{\control} (t)$ is the optimal control process.

\begin{rmk}
 The above result can be used to obtain the solution to the classical problem of mean-variance portfolio optimization.  Suppose we wish to find an admissible portfolio process which minimizes var$(X(\timehorizon)) = \expectation \left( X(\timehorizon) - \expectation ( X(\timehorizon) ) \right)^{2}$ subject to $\expectation ( X(\timehorizon) ) = a$, for some $a \in \realnumbers$.  Applying a Lagrange multiplier technique, we note that for all $\lambda \in \realnumbers$,
\begin{displaymath}
 \expectation \left( \left( X(\timehorizon) - a \right)^{2} + 2 \lambda ( X(\timehorizon) - a) \right) = \expectation \left( X(\timehorizon) - a + \lambda \right)^{2} - \lambda^{2}.
\end{displaymath}
Fix $\lambda \in \realnumbers$ and minimize $\expectation \left( X(\timehorizon) - a + \lambda \right)^{2}$.  The portfolio process which minimizes this is $\hat{\control} (t):= \hat{\control} (t;\lambda)$, which is given by (\ref{EQNoptimalcontrol}) with $\constc := a - \lambda$.  Then we maximize the quadratic function $\expectation \left( X(\timehorizon) - a + \lambda \right)^{2} - \lambda^{2}$ over all $\lambda \in \realnumbers$ to find the optimal $\lambda^{\star} \in \realnumbers$ and hence we obtain the optimal portfolio process $\hat{\control} (t;\lambda^{\star})$ which solves the mean-variance problem.
\end{rmk}

\begin{rmk}
 The optimal control process for the mean-variance problem was also found in \citet{zhouyin03.article} using a stochastic LQ control technique  and completion-of-squares.
\end{rmk}

\section*{Acknowledgments}
This work was carried out while the author was at ETH Zurich, Switzerland.  The author thanks RiskLab, ETH Zurich, Switzerland for financial support and an anonymous referee for constructive comments which improved the paper.

\bibliographystyle{plainnat}

\bibliography{article}

\end{document}